\documentclass[graybox,envcountsame,envcountresetchap]{svmult}

\usepackage{graphicx}

\usepackage{amsmath}
\usepackage{amssymb}
\usepackage{color}

\usepackage[latin1]{inputenc}
\usepackage{amsfonts}
\usepackage{tikz} 
\usetikzlibrary{matrix}
\usepackage[english]{babel}
\usepackage[latin1]{inputenc}
\usepackage[T1]{fontenc}
\usepackage{ae}

\usepackage{url}

\allowdisplaybreaks[4]

\DeclareMathOperator{\sign}{sign}
\DeclareMathOperator{\GL}{GL}

\usepackage{color}
\usepackage{tikz}
\usetikzlibrary{matrix}
\usepackage{cite}           
\usepackage{mathptmx}       
\usepackage{helvet}         
\usepackage{courier}        
\usepackage{type1cm}        
\usepackage{makeidx}         
\usepackage{graphicx}        
\usepackage{multicol}        
\usepackage[bottom]{footmisc}

\newcommand{\C}{\mathbb{C}}

\newcommand{\Q}{\mathbb{Q}}
\newcommand{\Z}{\mathbb{Z}}
\newcommand{\z}{\zeta}
\newcommand{\al}{\alpha}
\newcommand{\be}{\beta}
\newcommand{\ga}{\gamma}
\newcommand{\la}{\lambda}
\newcommand{\g}{{\mathfrak g}}
\newcommand{\de}{\delta}
\renewcommand{\th}{\theta}

\newcommand{\G}{\Gamma}

\newcommand{\psmm}[4]{\left(\begin{smallmatrix}{#1}&{#2}\\{#3}&{#4}\end{smallmatrix}\right)}

\DeclareMathOperator{\lcm}{lcm}

\renewcommand{\E}{{\mathcal E}}
\renewcommand{\H}{\mathfrak H}
\newcommand{\M}{\mathcal M}

\DeclareMathOperator{\Res}{Res}

\newcommand{\ov}{\overline}
\newcommand{\leg}[2]{\mbox{$\left(\dfrac{#1}{#2}\right)$}}

\newcommand{\fp}{\qed}

\makeatletter
\def\renewtheorem#1{%
  \expandafter\let\csname#1\endcsname\relax
  \expandafter\let\csname c@#1\endcsname\relax
  \gdef\renewtheorem@envname{#1}
  \renewtheorem@secpar
}
\def\renewtheorem@secpar{\@ifnextchar[{\renewtheorem@numberedlike}{\renewtheorem@nonumberedlike}}
  \def\renewtheorem@numberedlike[#1]#2{\newtheorem{\renewtheorem@envname}[#1]{#2}}
  \def\renewtheorem@nonumberedlike#1{
    \def\renewtheorem@caption{#1}
    \edef\renewtheorem@nowithin{\noexpand\newtheorem{\renewtheorem@envname}{\renewtheorem@caption}}
    \renewtheorem@thirdpar
  }
  \def\renewtheorem@thirdpar{\@ifnextchar[{\renewtheorem@within}{\renewtheorem@nowithin}}
    \def\renewtheorem@within[#1]{\renewtheorem@nowithin[#1]}
    \makeatother

\renewtheorem{theorem}{Theorem}[section]
\renewtheorem{remarks}[theorem]{{\it Remarks}}
\smartqed

\makeindex             

\begin{document}

\title*{Expansions at Cusps and Petersson Products in Pari/GP}
\author{Henri Cohen}
\institute{Universit\'e de Bordeaux,\\
  Institut de Math\'ematiques de Bordeaux (IMB)\\
  UMR 5251 du CNRS, Equipe LFANT INRIA,\\
  351 Cours de la Lib\'eration, 33405 Talence Cedex\\
\email{Henri.Cohen@math.u-bordeaux.fr}}
%
%
\maketitle

\abstract{We begin by explaining how to compute Fourier expansions at all
  cusps of any modular form of integral or half-integral weight thanks to
  a theorem of Borisov--Gunnells and explicit expansions of Eisenstein
  series at all cusps. Using this, we then give a number of
  methods for computing arbitrary Petersson products. All this is available in
  the current release of the {\tt Pari/GP} package.}

%
%

\section{Introduction}\label{sec:1}

\vspace{1mm}\noindent

In this paper we consider the practical problem of numerically computing
Petersson products of two modular forms whenever it is defined.
In some cases this can be done using the Rankin--Selberg convolution of
the forms, but in general this is not always possible nor practical.

We will describe three methods. The first is applicable when both forms
are cusp forms, and is a variant of the well-known formulas of Haberland.
The second is a modification of the first, necessary when at least one
of the forms is not a cusp form. Both of these methods need the essential
condition that the weight $k$ be integral and greater than or equal to $2$.
The third method is due to P.~Nelson and D.~Collins. It has the great
advantage of being also applicable when $k=1$ or $k$ half-integral, but
the great disadvantage of being much slower when $k$ is integral and
greater than or equal to $2$.

All of these methods require the possibility of computing the Fourier
expansion of $f|_k\ga$ for an arbitrary $\ga$ in the full modular group.
The method used in {\tt Pari/GP} is to express any modular form (possibly
multiplied by a known Eisenstein or theta series) as a linear combination
of products of \emph{two} Eisenstein series, which is always possible
thanks to a theorem of Borisov--Gunnells \cite{Bor-Gun} \cite{Bor-Gun2},
so we will begin by studying this in detail here, so that the formulas can be
recorded.

\section{Eisenstein Series}

\subsection{Introduction}

In the sequel, we let $\chi_1$ and $\chi_2$ be two \emph{primitive}
characters modulo $N_1$ and $N_2$ respectively. For $k\ge3$ we define
$$G_k(\chi_1,\chi_2)(\tau)=\dfrac{1}{2}\sum_{N_1\mid c}\dfrac{\ov{\chi_1(d)}\chi_2(c/N_1)}{(c\tau+d)^k}\;,$$
and for $k=2$ and $k=1$ we define $G_k$ by analytic continuation to $s=0$
of the same sum with an extra factor $|c\tau+d|^{-2s}$ (Hecke's trick).
We will always assume that $\chi_1\chi_2(-1)=(-1)^k$, otherwise the series
is identically zero.

If $k\ne2$ or $k=2$ and $\chi_1$ and $\chi_2$ are not both trivial, then
$G_k\in M_k(\G_0(N_1N_2),\chi_1\chi_2)$ (if $k=2$ and $\chi_1$ and $\chi_2$
are both trivial we have a nonanalytic term in $1/\Im(\tau)$). The Fourier
expansion at infinity is given by
$$G_k(\chi_1,\chi_2)(\tau)=\left(\dfrac{-2\pi i}{N_1}\right)^k\dfrac{\g(\ov{\chi_1})}{(k-1)!}F_k(\chi_1,\chi_2)(\tau)\;,$$
where $\g(\ov{\chi})$ is the standard Gauss sum associated to $\ov{\chi}$,
$$F_k(\chi_1,\chi_2)(\tau)=\de_{N_2,1}\dfrac{L(\chi_1,1-k)}{2}+\sum_{n\ge1}\sigma_{k-1}(\chi_1,\chi_2,n)q^n\;,$$
where $\de$ is the Kronecker delta, and
$$\sigma_{k-1}(\chi_1,\chi_2,n)=\sum_{d\mid n,\ d>0}d^{k-1}\chi_1(d)\chi_2(n/d)\;.$$
By convention, we will set $F_0=1$.

An important theorem of Borisov--Gunnells \cite{Bor-Gun} \cite{Bor-Gun2}
says that in weight
$k\ge3$, and very often also in weight $2$, any modular form
$f\in M_k(\G_0(N),\chi)$ is a linear combination of
$F_{\ell}(\chi_1,\chi_2)(e\tau)F_{k-\ell}(\chi'_1,\chi'_2)(e'\tau)$ for
suitable characters $\chi$, $\ell$, $e$ and $e'$.

If we are in the unfavorable case of the theorem (only in weight $2$), or
in weight $1$, we can simply multiply by a known Eisenstein series
(of weight $1$ or $2$) to be in a case where the theorem applies. Similarly,
if we are in half-integral weight, we simply multiply by a suitable power
of $\th\in M_{1/2}(\G_0(4))$ to be able to apply the theorem.

For us, the main interest of this theorem is that the Fourier expansion of
$F_k|_k\ga$ as well as that of $\th|_{1/2}\ga$ can be explicitly computed for
all $\ga\in\G$, the full modular group, so this allows us to compute
$f|_k\ga$ for any modular form $f$, and in particular find the Fourier
expansions at any cusp.

\subsection{Expansion of $F_k|_k\ga$}

As usual we denote by $N_1$ and $N_2$ the conductors of $\chi_1$ and $\chi_2$.
For simplicity of notation we will set
$F_k(\chi_1,\chi_2,e)(\tau):=F_k(\chi_1,\chi_2)(e\tau)$ and
$N=N_1N_2e$, so that $F_k(\chi_1,\chi_2,e)\in M_k(\G_0(N),\chi_1\chi_2)$.
Note that in the application using the Borisov--Gunnells theorem $N$ will
only be a \emph{divisor} of the level.

We now let $\ga\in \GL_2^+(\Q)$ be any matrix with rational coefficients
and strictly positive determinant. We want to compute the Fourier
expansion at infinity of $F_k(\chi_1,\chi_2,e)|_k\ga$. For this, we first
make three reductions. First, trivially the action of $\ga$ is homogeneous,
so possibly after multiplying $\ga$ by a common denominator we may assume
that $\ga=\psmm{A}{B}{C}{D}\in M_2^+(\Z)$. Second, by Euclid we can
find integers $u$, $v$, and $g$ such that $\gcd(A,C)=g=uA+vC$, and we have
the matrix identity
$$\begin{pmatrix}A&B\\C&D\end{pmatrix}=\begin{pmatrix}A/g&-v\\C/g&u\end{pmatrix}\begin{pmatrix}g&uB+vD\\0&(AD-BC)/g\end{pmatrix}\;,$$
where we note that the first matrix is in $\G$. Since the second one is
upper triangular, its action on a Fourier expansion is trivial to write
down, so we are reduced to the case where $\ga\in\G$.

The third and last reduction is based on the following easy lemma:

\begin{lemma} Let $\ga\in\G$. There exist $\be\in\G_0(N)$ and $m\in\Z$
  such that $$\ga=\be\begin{pmatrix}A&B\\C&D\end{pmatrix}T^m$$
  with $C\mid N$, $C>0$, and $N\mid B$.\end{lemma}

Since the action of $\be=\psmm{a}{b}{c}{d}\in\G_0(N)$ on $M_k(\G_0(N),\chi)$
is simply multiplication by $\chi(d)$, and since once again the action of the
translation $T^m$ is trivial to write down on Fourier expansions, this lemma
allows us to reduce to $\ga\in\G$ with the additional conditions $C\mid N$,
$C>0$, and $N\mid B$.

To state the main result we need to introduce an additional function needed
to express the constant terms:

\begin{definition} Let $\chi$ be a Dirichlet character modulo $M$, let
  $f$ be its conductor, and let $\chi_f$ be the primitive character modulo
  $f$ equivalent to $\chi$. We define
  $$S_k(\chi)=(M/f)^k\g(\chi_f)\dfrac{\ov{B_k(\chi_f)}}{k}\prod_{p\mid N}\left(1-\dfrac{\chi_f(p)}{p}\right)\;,$$
  where as usual $B_k(\chi_f)$ is the $\chi_f$-Bernoulli number.
\end{definition}

Note that $S_k(\chi)=-(2(k-1)!M^k/(-2\pi i)^k)L(\chi,k)$, but we have preferred
to give it in the above form to emphasize the fact that it belongs to a
specific cyclotomic field.

We are now ready to state the main result, where we always use the
convention $q^x=e^{2\pi i\tau x}$ when $x\in\Q$:
  
\begin{theorem} Set $N=eN_1N_2$, and let $\ga=\psmm{A}{B}{C}{D}\in\G$ be such
that $C\mid N$, $C>0$, and $N\mid B$. Set $g=\gcd(e,C)$, $g_1=\gcd(N_1g,C)$,
and $g_2=\gcd(N_2g,C)$. If $(k,\chi_1,\chi_2)\ne(2,1,1)$ we have
$$F_k(\chi_1,\chi_2,e)|_k\ga=\dfrac{1}{z_k(\chi_1,\chi_2,C)}\sum_{n\ge0}a_{\ga}(n)q^{g_1g_2n/N}\;,$$
where
\begin{enumerate}\item
$$z_k(\chi_1,\chi_2,C)=2(N_2e/g_2)^{k-1}(e/g)\g(\ov{\chi_1})\g(\ov{\chi_2})\;,$$
\item For $n\ge1$
$$a_{\ga}(n)=\z_N^{A^{-1}(g_1g_2/C)n}\sum_{m\mid n,\ m\in\Z}\sign(m)m^{k-1}c(n,m)\;,\text{\quad with}$$
\begin{align*}
c(n,m)&=\sum_{\substack{s_1\bmod C/g\\(N_1g/g_1)s_1\equiv n/m\pmod{C/g_1}}}\ov{\chi_1}((n/m-(N_1g/g_1)s_1)/(C/g_1))\cdot\\
&\phantom{=}\cdot\sum_{\substack{s_2\bmod{C/g}\\(N_2g/g_2)s_2\equiv m\pmod{C/g_2}}}\ov{\chi_2}((m-(N_2g/g_2)s_2)/(C/g_2))\z_{C/g}^{-(Ae/g)^{-1}s_1s_2}\;.\end{align*}
\item Set
$$T_k(\chi_1,\chi_2)=\begin{cases}(-1)^{k-1}\dfrac{\g(\ov{\chi_2})}{N_2(g_2/g)^{k-1}}\ov{\chi_1}(-Ae/g)S_k(\ov{\chi_1}\chi_2)&\text{\quad if $C/g=N_1$\;,}\\
0&\text{\quad if $C/g\ne N_1$\;.}\end{cases}$$
We have
$$a_{\ga}(0)=\begin{cases}T_k(\chi_1,\chi_2)&\text{\quad if $k>1$\;,}\\
T_1(\chi_1,\chi_2)+T_1(\chi_2,\chi_1)&\text{\quad if $k=1$\;.}
\end{cases}$$
\end{enumerate}\end{theorem}

Note that $(k,\chi_1,\chi_2)=(2,1,1)$ corresponds to the quasimodular
form $F_2$ (or $E_2$) which can be easily treated directly thanks to the
first matrix identity given above applied to $\ga=\psmm{eA}{eB}{C}{D}$.

\subsection{Rationality Questions}

To use this theorem in algorithmic practice, we need to make a choice.
As can be seen on the expression of $a_{\ga}(n)$, the coefficients of the
expansion belong to the large cyclotomic field $\Q(\z_N,\z_{\phi(N)})$,
which is in fact also the field which contains Gauss sums of characters
modulo $N$. When $N$ is not tiny, say when $N$ is a prime around $1000$,
this is a very large number field, so it seems almost impossible to work
with exact elements of the field. In our implementation we thus have chosen
to work with approximate complex values having hopefully sufficient accuracy
(note that this is sufficient in the application to Petersson products).
At the end of the computation of $f|_k\ga$ we may however want to recover
the exact algebraic values. This can of course be done using LLL-type
algorithms, with an a priori guess of the field of coefficients. But this
can be done rigorously by using the following results.

\smallskip

First, assume that $\gcd(N/\gcd(N,C),C)=1$, or equivalently
$(N,C^2)=(N,C)$ (so that the cusp $A/C$ will be regular). We then have the
following two results:

\begin{lemma} Assume that $\gcd(N/\gcd(N,C),C)=1$, and set $g=\gcd(N,C)$ and
$Q=N/g$. There exist an Atkin--Lehner matrix of the form
$W_Q=\psmm{Qx}{y}{N}{Q}$, a matrix $\de=\psmm{a}{b}{c}{d}\in\G_0(N)$, and an
integer $v$, such that
$$\ga=W_Q\de\psmm{1/Q}{v/Q}{0}{1}\;,$$
and we have $d\equiv Q^{-1}D\pmod{N/Q}$, where $Q^{-1}$ is an inverse of $Q$
modulo $C$.\end{lemma}

As usual, since the action of $\de$ on $M_k(\G_0(N),\chi)$ is multiplication
by $\chi(d)$ and the action of an upper triangular matrix on Fourier expansions
is trivial to write, we are reduced to computing the field of coefficients
of $f|_kW_Q$. This is given by a theorem essentially due to
Shimura and Ohta, and extended to cover half-integral weight as well.
We first define a normalizing constant $C(k,\chi,W_Q)$ as follows. First
recall that if $\gcd(Q,N/Q)=1$, which is the case here, we can write in
a unique way $\chi=\chi_Q\chi_{N/Q}$ with $\chi_Q$ defined modulo $Q$ and
$\chi_{N/Q}$ modulo $N/Q$.

\begin{definition} Let $W_Q=\psmm{Qx}{y}{Nz}{Qt}$ be a general Atkin--Lehner
matrix.
\begin{enumerate}\item We set $s(k,W_Q)=1$ unless $k$ is a half integer,
in which case we set $s(k,W_Q)=i^{(x-1)/2}$ if $Q$ is odd, and
$s(k,W_Q)=1+(-1)^{k+y/2}i$ if $4\mid Q$ (note that we cannot have
$Q\equiv2\pmod4$).
\item We define $C(k,\chi,W_Q)=s(k,W_Q)/(\g(\chi_Q)Q^{k/2})$.
\end{enumerate}
\end{definition}

The theorem is as follows:

\begin{theorem} Let $F\in M_k(\G_0(N),\chi)$ with $k$ integral or half
integral, set $K=\Q(F)$, let $Q\Vert N$ be a primitive divisor of $N$, and
let $W_Q=\psmm{Qx}{y}{Nz}{Qt}$ be a general Atkin--Lehner matrix.
We have $\Q(C(k,\chi,W_Q)F|_kW_Q)\subset K$.\end{theorem}

\smallskip

In the general case we cannot use Atkin--Lehner involutions, but again
using the Borisov--Gunnells theorem F.~Brunault and M.~Neururer recently
proved the following theorem, and we thank them for permission to include
it here. Their proof is given in the appendix to this paper.

\begin{theorem} Let $\ga=\psmm{A}{B}{C}{D}\in\G$, denote by $M\mid N$ the
conductor of $\chi$, and as in the previous theorem set $K=\Q(F)$. If
$f\in M_k(\G_0(N),\chi)$ with $k$ integral the Fourier coefficients of
$f|_k\ga$ belong to the cyclotomic extension $K(\zeta_R)$, where
$R=\lcm(N/\gcd(N,CD),M/\gcd(M,BC))$.\end{theorem}

\subsection{Expansion of $\th|_{1/2}\ga$}

For completeness, we also give the expansion of $\th|_{1/2}\ga$ which is
needed in the half-integral weight case. Thanks to the first two reductions
above (the third is not necessary) we may assume that
$\ga=\psmm{A}{B}{C}{D}\in\G$.
We recall that the \emph{theta multiplier} $v_{\th}(\ga)$ is given by
$$v_{\th}(\ga)=\leg{-4}{D}^{-1/2}\leg{C}{D}\;,$$
where we always choose the principal branch of the square root.
The result is then as follows:

\begin{proposition}\begin{enumerate}
  \item If $4\mid C$ we have
    $$\th|_{1/2}\ga=v_{\th}(\ga)\th=v_{\th}(\ga)\left(1+2\sum_{n\ge1}q^{n^2}\right)\;.$$
  \item If $C\equiv2\pmod4$, set $\al=\psmm{A-2B}{B}{C-2D}{D}$. Then
    $$\th|_{1/2}\ga=2v_{\th}(\al)\sum_{n\ge0}q^{(2n+1)^2/4}\;.$$
  \item If $2\nmid C$, let $\la\equiv-D/C\pmod4$, and set $D'=D+\la C$,
    $B'=\la A$, and $\al=\psmm{-B'}{A}{-D'}{C}$. Then
    $$\th|_{1/2}\ga=\dfrac{1-i}{2}v_{\th}(\al)\left(1+2\sum_{n\ge1}i^{-\la n^2}q^{n^2/4}\right)\;.$$
  \end{enumerate}
\end{proposition}

\subsection{Fourier Expansion of $f|_k\ga$}\label{fourga}

We need to recall some notation relative to the Fourier expansion of
$f|_k\ga$ for $\ga\in\G$ and $f\in M_k(\G_0(N),\chi)$. It is easy to show
that it has the form
$$f|_k\ga(\tau)=q^{\al(\ga)}\sum_{n\ge0}a_{\ga}(n)q^{n/w(\ga)}\;,$$
where $w(\ga)$ is the \emph{width} of the cusp $\ga(i\infty)$ and $\al(\ga)$
is a rational number in $[0,1[$, which by definition is different from $0$ if
and only the cusp is \emph{irregular}. For $\ga=\psmm{A}{B}{C}{D}$, these
quantities are given by the formulas
$$w(\ga)=\dfrac{N}{\gcd(N,C^2)}\text{\quad and\quad}e^{2\pi i\al(\ga)}=\chi\left(1+\dfrac{ANC}{\gcd(N,C^2)}\right)=\chi(1+ACw(\ga))\;.$$
In addition, note that the denominator of $\al(\ga)$ divides
$\gcd(N,C^2)/\gcd(N,C)$, and that $w(\ga)$ and $\al(\ga)$ only depend on the
representative $c$ of the cusp $\ga(i\infty)=A/C$, so we will denote them
$w(c)$ and $\al(c)$.

\subsection{Computation of all $f|_k\ga_j$}\label{gajall}

In the application to Petersson products we will need to compute \emph{all}
the Fourier expansions of $f|_k\ga_j$ for a system of right coset
representatives of $\G_0(N)\backslash\G$, i.e., such that
$\G=\bigsqcup_{j=1}^r\G_0(N)\ga_j$. Although the formulas that we will give
are independent of this choice, for efficiency reasons it is essential to
do it properly.

Let $C$ be a set of representatives of cusps of $\G_0(N)$ (which is
\emph{much} smaller than the set of cosets: for instance if $N$ is prime we
have $N+1$ cosets but only $2$ cusps), and for each $c\in C$ let $\ga_c\in\G$
such that $\ga_c(i\infty)=c$, and as above let $w(c)$ be the width of the
cusp $c$. We claim that the $(\ga_cT^m)_{c\in C,\ 0\le m<w(c)}$ form a system
of right coset representatives of $\G_0(N)\backslash\G$. Indeed, let
$\ga\in\G$, and let $c$ be the representative of the cusp $\ga(i\infty)$.
By definition this means that there exists $\de\in\G_0(N)$ such that
$\ga(i\infty)=\de(c)=\de\ga_c(i\infty)$, so $\ga=\de\ga_cT^m$ for some integer
$m$, and by definition of the width $\ga_cT^{w(c)}\ga_c^{-1}\in\G_0(N)$,
so we can always reduce $m$ modulo $w(c)$, proving our claim since
$\sum_{c\in C}w(c)=[\G:\G_0(N)]$.

Thus, we simply compute
$$f|_k\ga_c(\tau)=q^{\al(c)}\sum_{n\ge0}a_{\ga_c}(n)q^{n/w(c)}\;,$$
and we deduce that
$$f|_k(\ga_cT^m)(\tau)=e^{2\pi im\al(c)}q^{\al(c)}\sum_{n\ge0}a_{\ga_c}(n)\z_{w(c)}^{nm}q^{n/w(c)}$$
with $\z_{w(c)}=e^{2\pi i/w(c)}$, so we only need to compute $|C|$ expansions
and not $[\G:\G_0(N)]$.

\section{Petersson Products: Haberland-type Formulas}

Now that we know how to compute the Fourier expansion of $f|_k\ga$ for any
$\ga\in\G$ (and even $\ga\in\GL_2^+(\Q)$), we apply this to the computation
of Petersson products.

\subsection{Preliminary Formulas}

Although this has been explained in several places, for instance in
\cite{ANTS} and \cite{Coh-Str}, it is necessary to reproduce the statements
and proofs, since we will need some important modifications. In this
section we always assume that $k$ is an integer such that $k\ge2$, so that
$(X-\tau)^{k-2}$ is a polynomial.

In what follows, $f$ and $g$ will denote two modular forms in the space
$M_k(\G_0(N),\chi)$, and as above we denote by $(\ga_j)_{1\le j\le r}$ a set
of right coset representatives of the full modular group $\G$ modulo $\G_0(N)$,
so that $\G=\bigsqcup_{j=1}^r\G_0(N)\ga_j$. Finally, we set $f_j=f|_k\ga_j$
and $g_j=g|_k\ga_j$.

It is clear that for any $\al\in\G$ there exist an index
which by abuse of notation we will write as $\al(j)$, and an
element $\de_j(\al)\in\G_0(N)$ such that $\ga_j\al=\de_j(\al)\ga_{\al(j)}$,
and the map $j\mapsto\al(j)$ is a bijection of $[1,r]$.

\begin{definition} For any $j\in[1,r]$ and $Z_j\in\ov{\H}$ we define
  $$G_j(Z_j;\tau)=\int_{Z_j}^\tau\ov{g_j(\tau_2)}(\tau-\ov{\tau_2})^{k-2}\,d\ov{\tau_2}\;.$$\end{definition}

Note that this function is essentially an \emph{Eichler integral} of $g_j$, so
will have quasi-modularity properties in weight $2-k$. More precisely:

\begin{proposition} Keep the above notation. We have
  $$(G_j(Z_j;\tau)|_{2-k}\al)(\tau)=\ov{\chi(\de_j(\al))}\left(G_{\al(j)}(Z_{\al(j)};\tau)-P_{\al(j)}(\al;\tau)\right)\;,$$
  where as usual $\chi\left(\psmm{a}{b}{c}{d}\right)=\chi(d)$, and $P_j$ is
  the \emph{polynomial} in $\tau$
  $$P_j(\al;\tau)=\int_{Z_j}^{\al^{-1}\left(Z_{\al^{-1}(j)}\right)}\ov{g_j(\tau_2)}(\tau-\ov{\tau_2})^{k-2}\,d\ov{\tau_2}\;.$$
\end{proposition}

\begin{corollary} Keep the notation of the proposition. For any $A$ and $B$ in
  $\ov{\H}$ we have
  $$\left(\int_A^B-\int_{\al(A)}^{\al(B)}\right)\sum_{1\le j\le r}f_j(\tau)G_j(Z_j;\tau)\,d\tau=\int_A^B\sum_{1\le j\le r}f_j(\tau)P_j(\al;\tau)\,d\tau\;.$$
\end{corollary}

The main theorem proved for instance in \cite{Coh-Str}, but which is
an immediate consequence of Stokes's theorem, is the following:

\begin{theorem} Let $H$ be some subgroup of $\G$ of finite index $s=[\G:H]$,
  and let $D(H)$ denote a fundamental domain for $H$ whose boundary
  $\partial(D(H))$ is a hyperbolic polygon. Then for any choice of the $Z_j$
  we have
  $$rs(2i)^{k-1}<f,g>_{\G_0(N)}=\int_{\partial(D(H))}\sum_{1\le j\le r}f_j(\tau)G_j(Z_j;\tau)\,d\tau\;.$$
\end{theorem}

Note that the subgroup $H$ can be chosen arbitrarily. To simplify, we will
choose it so that $\partial(D(H))$ is a hyperbolic quadrilateral
$(A_1,A_2,A_3,A_4)$ such that there exist an element $\al_1\in \G$ sending
$[A_1,A_2]$ to $[A_3,A_2]$ and $\al_2\in\G$ sending $[A_3,A_4]$ to
$[A_1,A_4]$. We thus have
$$\int_{\partial(D(H))}=\left(\int_{A_1}^{A_2}-\int_{\al_1(A_1)}^{\al_1(A_2)}\right)+\left(\int_{A_3}^{A_4}-\int_{\al_2(A_3)}^{\al_2(A_4)}\right)\;.$$
Applying the above corollary and the theorem we deduce the following.

\begin{definition} The forms $f$ and $g$ being implicit, we define
  $$G_j(A,B;C,D)=\int_A^B\int_C^Df_j(\tau)\ov{g_j(\tau_2)}(\tau-\ov{\tau_2})^{k-2}\,d\tau\,d\ov{\tau_2}\;.$$
\end{definition}

\begin{corollary}\label{cortmp1} We have
  $$rs(2i)^{k-1}<f,g>_{\G_0(N)}=\sum_{1\le j\le r}(I_j(f,g)+J_j(f,g))\;,$$
  where
  \begin{align*}
    I_j(f,g)&=G_j\left(A_1,A_2;Z_j,\al_1^{-1}\left(Z_{\al_1^{-1}(j)}\right)\right)\\
    J_j(f,g)&=G_j\left(A_3,A_4;Z_j,\al_2^{-1}\left(Z_{\al_2^{-1}(j)}\right)\right)\;.
  \end{align*}
\end{corollary}

\subsection{The Cuspidal Case}

We now distinguish whether both $f$ and $g$ are cusp forms or otherwise.

Assume first that $f$ and $g$ are both cusp forms. As in \cite{Coh-Str}
we choose $H=\G(2)$, which has index $6$ in $\G$, and we can take for
$D(H)$ the hyperbolic quadrilateral with $A_1=1$, $A_2=i\infty$,
$A_3=-1$, and $A_4=0$, so that $\al_1=T^{-2}=\psmm{1}{-2}{0}{1}$ and
$\al_2=\psmm{1}{0}{2}{1}$. We also choose $Z_j=0$ for all $j$, so that
$\al_2^{-1}(\Z_{\al_2^{-1}(j)})=\al_2^{-1}(0)=0$, hence $P_j(\al_2;\tau)=0$,
so that $J_j(f,g)=0$ for all $j$. On the other hand
$$P_j(\al_1;\tau)=\int_0^2\ov{g_j(\tau_2)}(\tau-\ov{\tau_2})^{k-2}\,d\ov{\tau_2}\;,$$
so that
$$6r(2i)^{k-1}<f,g>_{\G_0(N)}=\sum_{1\le j\le r}G_j(1,i\infty;0,2)\;.$$
Shifting both $\tau$ and $\tau_2$ by $1$ gives the following:

\begin{corollary}\label{corhab} Assume that $f$ and $g$ are both cusp forms.
  We then have
  \begin{align*}6r(2i)^{k-1}<f,g>_{\G_0(N)}&=\sum_{1\le j\le r}G_j(0,i\infty;-1,1)\\
    &=\sum_{1\le j\le r}\int_0^{i\infty}\int_{-1}^1f_j(\tau)\ov{g_j(\tau_2)}(\tau-\ov{\tau_2})^{k-2}\,d\tau\,d\ov{\tau_2}\\
    &=\sum_{1\le j\le r}\sum_{0\le n\le k-2}(-1)^n\binom{k-2}{n}I_{k-2-n}(0,i\infty,f_j)\ov{I_n(-1,1,g_j)}\;,\end{align*}
  where we have set
  $$I_n(A,B,f)=\int_A^B\tau^nf(\tau)\,d\tau\;.$$
\end{corollary}

The essential advantage of this formula is that we have reduced the computation
of a Petersson product, which is a double integral, to a small finite
number of single integrals, which are essentially the periods associated to
$f$ and $g$; this is in fact exactly the statement of Haberland's theorem.

The main problem is that, even though the Petersson product is defined when
only one of $f$ and $g$ is a cusp form, we cannot apply the above formula
since the period integrals will diverge for non cusp forms. We thus consider
the general case.

\subsection{The Noncuspidal but Convergent Case}

We now assume that $f$ and $g$ are in $M_k(\G_0(N),\chi)$, not necessarily
cusp forms. For the Petersson product to converge it is necessary and
sufficient that at each cusp either $f$ or $g$ vanishes. Equivalently,
for each $j\in[1,r]$ either $f_j$ or $g_j$ vanishes as $\tau\to i\infty$.
We denote by $E$ the subset of $j\in[1,r]$ such that $f_j$ vanishes as
$\tau\to i\infty$, so that if $j\notin E$ then $g_j$ vanishes as
$\tau\to i\infty$. Consider now $T=\psmm{1}{1}{0}{1}$ the usual translation
by $1$. As usual $\ga_jT=\de_j(T)\ga_{T(j)}$ for some bijection $j\mapsto T(j)$
and $\de_j(T)\in\G_0(N)$. Thus $f_j|_kT=\chi(\de_j(T))f|_{T(j)}$, and
it follows that both $E$ and its complement are stable by the bijection induced
by $T$.

For simplicity, we are going to choose $H=\G$, and as fundamental domain
the usual fundamental domain of the modular group, which has the advantage
of having a single cusp on its boundary. Thus as above, setting as usual
$\rho=e^{2\pi i/3}$, we have $A_1=\rho+1$, $A_2=i\infty$, $A_3=\rho$, and
$A_4=i$, with $\al_1=T^{-1}$ and $\al_2=S=\psmm{0}{-1}{1}{0}$.

We will choose $Z_j=i$ if $j\in E$ and $Z_j=i\infty$ if $j\notin E$. With
the notation of Corollary \ref{cortmp1} we have
$$I_j(f,g)=\int_{\rho+1}^{i\infty} f_j(\tau)P_j(T^{-1};\tau)\,d\tau\;,$$
with $$P_j(T^{-1};\tau)=\int_{Z_j}^{Z_{T(j)}+1}\ov{g_j(\tau_2)}(\tau-\ov{\tau_2})^{k-2}\,d\ov{\tau_2}\;.$$
Note that if $j\in E$ there is no convergence problem since $f_j(\tau)$ tends
to $0$ exponentially fast. On the other hand, if $j\notin E$ we have chosen
$Z_j=i\infty$, and we also have $T(j)\notin E$ by what we said above, so
$P_j(T^{-1};\tau)$ vanishes in that case. We thus have
$$\sum_{1\le j\le r}I_j(f,g)=\sum_{j\in E}G_j(\rho+1,i\infty;i,i+1)\;,$$
and we can again expand this by the binomial theorem as a linear combination
of products of two simple integrals, since the integral of $f_j(\tau)$
converges at $i\infty$.

Similarly, we have
$$J_j(f,g)=\int_{\rho}^i f_j(\tau)P_j(S;\tau)\,d\tau\;,$$
with
$$P_j(S;\tau)=\int_{Z_j}^{-1/Z_{S(j)}}\ov{g_j(\tau_2)}(\tau-\ov{\tau_2})^{k-2}\,d\ov{\tau_2}\;.$$
Here we must distinguish four cases.

\begin{enumerate}
\item If $j\in E$ and $j\in S(E)$ (or equivalently $S(j)\in E$)
  then $Z_j=i$ and $Z_{S(j)}=i$ so $-1/Z_{S(j)}=i$, hence $P_j(S;\tau)=0$.
\item If $j\in E$ and $j\notin S(E)$, so that $S(j)\notin E$ we have
  $Z_j=i$, $Z_{S(j)}=i\infty$, so $P_j$ is an integral from $i$ to $0$
  hence $J_j(f,g)=G_j(\rho,i;i,0)$. Note that since $S(j)\notin E$ by
  assumption $g_{S(j)}$ vanishes at $i\infty$, or equivalently $g_j$ vanishes
  at $0$ so the integral makes sense.
\item If $j\notin E$ and $j\in S(E)$, we have $Z_j=i\infty$, $Z_{S(j)}=i$,
  so $J_j(f,g)=G_j(\rho,i;i\infty,i)=-G_j(\rho,i;i,i\infty)$.
\item If $j\notin E$ and $j\notin S(E)$, we have $Z_j=i\infty$,
  $Z_{S(j)}=i\infty$, so $J_j(f,g)=G_j(\rho,i;i\infty,0)$.
\end{enumerate}

The changes of variable $\tau\mapsto S(\tau)$ and $\tau_2\mapsto S(\tau_2)$
show that $G_j(\rho,i;i,0)=G_{S(j)}(\rho+1,i;i,i\infty)$. Thus
$$\sum_{j\in E,\ j\notin S(E)}J_j(f,g)=\sum_{j\notin E,\ j\in S(E)}G_j(\rho+1,i;i,i\infty)\;.$$
Combining with (3), it follows by transitivity that
$$\left(\sum_{j\in E,\ j\notin S(E)}+\sum_{j\notin E,\ j\in S(E)}\right)J_j(f,g)=\sum_{j\notin E,\ j\in S(E)}G_j(\rho+1,\rho;i,i\infty)\;.$$
We have thus shown the following:

\begin{theorem}\label{petnoncusp} We have
  $$r(2i)^{k-1}<f,g>_{\G_0(N)}=S_1+S_2+S_3$$
  with
  \begin{align*}
    S_1&=\sum_{j\in E}G_j(\rho+1,i\infty;i,i+1)\;,\\
    S_2&=\sum_{j\notin E,\ j\in S(E)}G_j(\rho+1,\rho;i,i\infty)\;,\\
    S_3&=\sum_{j\notin E,\ j\notin S(E)}G_j(\rho,i;i\infty,0)\;,
  \end{align*}
  and each $G_j$ can be expressed as a linear combination of products of
  two \emph{convergent} single integrals by using the binomial theorem.
\end{theorem}

Note that if $f$ is a cusp form we have $E=[1,r]$ and only $S_1$ contributes,
and if both $f$ and $g$ are cusp forms, we can either use this theorem or
the formula given in Corollary \ref{corhab}.

\subsection{Computation of Partial Periods}

In all of the above formulas, using the notation of Corollary \ref{corhab}
we need to compute integrals of the form $I_n(a,b,f_j)$ and $I_n(a,b,g_j)$
for specific values of $(a,b)$ in the completed upper half-plane. Putting
them together for $0\le n\le k-2$, this means that we must compute the
\emph{partial periods}
$$P(a,b,F)(X)=\int_a^b(X-\tau)^{k-2}F(\tau)\,d\tau$$ for $F=f_j$ and all $j$.
For future reference, note the following important but trivial identity:

\begin{lemma}\label{lempab} For any $\ga\in\G$ we have
$$P(a,b,F|_k\ga)(X)=P(\ga(a),\ga(b),F)|_{2-k}\ga(X)\;.$$
\end{lemma}

We also have the following immediate lemma:

\begin{lemma} Let $R_{k-2}(X)=\sum_{0\le n\le k-2}X^n/n!$ be the $(k-2)$nd
  partial sum of the exponential series. For all $m>0$ we have
  $$\int_a^{i\infty}(X-\tau)^{k-2}e^{2\pi mi\tau}\,d\tau=-e^{2\pi mia}\dfrac{(k-2)!}{(2\pi mi)^{k-1}}R_{k-2}(2\pi mi(X-a))\;.$$
\end{lemma}
  
We consider several cases. Keep in mind that in all the formulas that we
use for computing Petersson products the endpoints of integration are either
cusps or points in $\H$ with reasonably large imaginary part (at least
$\sqrt{3}/2$).

\begin{enumerate}
\item If $a\in\H$ and $b=i\infty$ (or the reverse), we write as usual
  $f_j(\tau)=\sum_{n\ge0}a_{\ga}(n)q^{\al(c)+n/w(c)}$ (where
  $c=\ga_j(i\infty)$), so that
  $$\int_a^b(X-\tau)^{k-2}f_j(\tau)\,d\tau=\sum_{n\ge0}a_{\ga}(n)\int_a^{i\infty}(X-\tau)^{k-2}e^{2\pi i(\al(c)+n/w(c))\tau}\,d\tau\;,$$
  and the inner integral is given by the lemma. The dominant term in the
  resulting series is $e^{2\pi (\al(c)+n/w(c))ia}$, so the convergence will be
  in $e^{-2\pi\Im(a)n/w(c)}$.
\item If $a\in\H$ and $b$ is a cusp (or the reverse), we choose $\ga\in\G$
  such that $b=\ga(i\infty)$, make the change of variable $\tau=\ga(\tau')$,
  and we are reduced to (1) with $f_{\ga(j)}$ instead of $f_j$.
\item If $a$ and $b$ are in $\H$ we simply write
  $\int_a^b=\int_a^{i\infty}-\int_b^{i\infty}$ and use (1).
\item If $a=0$ and $b=i\infty$ (or the reverse), we write the integral as
  $\int_0^{it_0}+\int_{it_0}^{i\infty}$. The second integral is treated as in
  (1), so with convergence in $e^{-2\pi t_0n/w(c)}$. In the first integral
  we make the change of variable $\tau\mapsto S(\tau)=-1/\tau$, and we
  again treat the resulting integral as in (1), with convergence in
  $e^{-2\pi(1/t_0)n/w(S(c))}$, where $w(S(c))$ is the width of the cusp
  $S(c)=\ga_j(S(i\infty))$. To optimize the speed, we thus choose
  $t_0=(w(c)/w(S(c)))^{1/2}$, so that the convergence of both integrals will be
  in $e^{-2\pi n/(w(c)w(S(c)))^{1/2}}$.
\item Finally, if $a$ and $b$ are both cusps, we use the well-known
  \emph{Manin decomposition} of a modular symbol as a sum of Manin symbols.
  More precisely, we proceed as follows. Write $a=A/C$ and $b=B/D$ with
  $\gcd(A,C)=\gcd(B,D)=1$. Then if $AD-BC=1$ we set $\ga=\psmm{A}{B}{C}{D}$,
  and using Lemma \ref{lempab} we transform our integral into an integral from
  $0$ to $i\infty$, so we apply (4) (similarly if $AD-BC=-1$).
  Otherwise, setting $\Delta=AD-BC$ and using $u$ and $v$ such that $uA+vC=1$,
  we write
  $$\begin{pmatrix}A&B\\C&D\end{pmatrix}=\begin{pmatrix}A&-v\\C&u\end{pmatrix}\begin{pmatrix}1&uB+vD\\0&\Delta\end{pmatrix}=\ga\begin{pmatrix}1&B'\\0&\Delta\end{pmatrix}$$
  for $B'=uB+vD$, where $\ga\in\G$. Let $(p_j/q_j)_{-1\le j\le m}$ be the
  convergents of the regular continued fraction expansion of $B'/\Delta$ with
  $p_{-1}/q_{-1}=1/0$ and $p_m/q_m=B'/\Delta$, and let $M_j$ be the matrix
  $$M_j=\begin{pmatrix}(-1)^{j-1}p_j&p_{j-1}\\(-1)^{j-1}q_j&q_{j-1}\end{pmatrix}\in\G\;.$$
  It is then immediate to show that
  $$P(A/C,B/D,F)(X)=\sum_{0\le j\le m}P(0,i\infty,F|_k(\ga M_j))|_{2-k}(\ga M_j)^{-1}(X)\;,$$
  so once again we can apply (4).
\end{enumerate}

Note that in Theorem \ref{petnoncusp} we need to use (1), (3), and (4),
while in Corollary \ref{corhab} we need to use (4) and (5), and in (5) we
have $(a,b)=(-1,1)$ so the Manin decomposition consists here simply in writing
$\int_{-1}^1=\int_{-1}^0+\int_0^1$, both integrals being then sent to
integrals from $0$ to $i\infty$ by suitable $\ga\in\G$.

\smallskip

In practice, the computation of these integrals forms only a very small
part of the computation time. Almost all of the time is spent in computing
the Fourier expansions at infinity of $f|_k\ga_j$, for instance using
products of two Eisenstein series as we do in this package. Note that there
is of course no need to \emph{rationalize} the expansions, and to compute all
these expansions at once we use the specific choice of the $\ga_j$ explained
in Section \ref{gajall}.

\section{Petersson Products: The Method of Nelson and Collins}

\vspace{1mm}\noindent

\subsection{The Basic Formula}

Recall that the completed zeta function $\Lambda(s)$ defined by
$\Lambda(s)=\pi^{-s/2}\G(s/2)\z(s)$ satisfies $\Lambda(1-s)=\Lambda(s)$.
Nelson's method, completed by Collins, is based on the following proposition,
essentially due to Rankin:

\begin{proposition} Let $F(\tau)$ be a bounded measurable function on $\H$
  invariant by the modular group $\G$, and such that for some fixed $\al>0$
  we have $F(x+iy)=O(y^{-\al})$ for almost all $\tau=x+iy$ with $y\ge1$.
  Denote by $a(0;F)(y)$ the constant term of the Fourier expansion of $F(\tau)$
  and by $\M(a(0;F))(s)=\int_0^\infty y^sa(0,F)(y)dy/y$ its Mellin
  transform. For any $\de>0$ we have
  $$\int_{\G\backslash\H}F(\tau)\,d\mu=\int_{\Re(s)=1+\de}(4s-2)\Lambda(2s)\M(a(0;F))(s-1)\,ds\;,$$
  where $d\mu=dxdy/y^2$ is the usual invariant hyperbolic measure.
\end{proposition}

\begin{proof} Recall that the standard nonholomorphic Eisenstein series of
  weight $0$ is defined by
  $E(s)=\sum_{\ga\in\G_\infty\backslash\G}\Im(\ga\tau)^s$,
  and its completed function $\E(s)=\Lambda(2s)E(s)$ satisfies
  $\E(1-s)=\E(s)$ and has only two poles, which are simple, at $s=0$ and
  $s=1$ with residues $-1/2$ and $1/2$ respectively. Standard unfolding
  shows that
  $$\int_{\G\backslash\H}E(s)(\tau)F(\tau)\,d\mu=\int_0^\infty y^{s-2}\int_0^1F(x+iy)\,dx\,dy\;.$$
  The inner integral is equal to $a(0;F)(y)$ so that
  $$\int_{\G\backslash\H}E(s)(\tau)F(\tau)\,d\mu=\M(a(0;F))(s-1)\;.$$
  On the other hand, by the residue theorem if $C_{\de}$ is the infinite
  vertical contour whose vertical sides are $\Re(s)=-\de$ and $\Re(s)=1+\de$,
  by the residue theorem we first have
  $$\dfrac{1}{2\pi i}\int_{C_{\de}}s\E(s)\,ds=\Res_{s=0}s\E(s)+\Res_{s=1}s\E(s)=1/2\;,$$
  and on the other hand, since $\E$ decreases exponentially when
  $|\Im(s)|\to\infty$ and $\E(1-s)=\E(s)$, this integral is equal to
  $$\dfrac{1}{2\pi i}\int_{\Re(s)=1+\de}(2s-1)\E(s)\,ds\;.$$
  Multiplying the resulting identity by $2F(\tau)$ and integrating on
  $\G\backslash\H$ gives
  $$\dfrac{1}{2\pi i}\int_{\Re(s)=1+\de}(4s-2)\Lambda(2s)\int_{\G\backslash\H}E(s)F(\tau)\,ds\,d\mu=\int_{\G\backslash\H}F(\tau)\,d\mu\;,$$
  hence
  $$\int_{\G\backslash\H}F(\tau)\,d\mu=\dfrac{1}{2\pi i}\int_{\Re(s)=1+\de}(4s-2)\Lambda(2s)\M(a(0;F))(s-1)\,ds\;,$$
  proving the proposition.\fp\end{proof}
  
\begin{corollary} Let $G$ be a subgroup of finite index of $\G$, let
  $C(G)$ be a system of representatives of the cusps of $G$, and for each
  $c\in C(G)$ let $\ga_c\in\G$ such that $\ga_c(i\infty)=c$. If $F(\tau)$
  is a bounded measurable function invariant by $G$ we have
  $$\int_{G\backslash\H}F(\tau)\,d\mu=\dfrac{1}{2\pi i}\int_{\Re(s)=1+\de}(4s-2)\Lambda(2s)\sum_{c\in C}w(c)\M(a(0;F|\ga_c))(s-1)ds\;,$$
  where $w(c)$ is the width of the cusp $c$.\end{corollary}

\begin{proof} Immediate by applying the proposition to
  $F_1=\sum_{\ga\in G\backslash\G}F|\ga$, noting that
  $\int_{\G\backslash\H}F_1(\tau)\,d\mu=\int_{G\backslash\H}F(\tau)\,d\mu$,
  and that $a(0;F|\ga'_c)=a(0;F|\ga_c)$ for any $\ga'_c$ such that
  $\ga'_c(i\infty)=\ga_c(i\infty)=c$.\fp\end{proof}

\subsection{Collins's Formula}

We are of course going to apply the above corollary to the function
$F(\tau)=f(\tau)\ov{g(\tau)}y^k$, with $y=\Im(\tau)$. Recall from
Section \ref{fourga} that we have expansions
$$f|_k\ga(\tau)=q^{\al(c)}\sum_{n\ge0}a_{\ga}(n)q^{n/w(c)}\text{\quad and\quad}  g|_k\ga(\tau)=q^{\al(c)}\sum_{n\ge0}b_{\ga}(n)q^{n/w(c)}$$
with the same $\al(c)$ and $w(c)$. It follows that the constant term
$a_{\ga}(0;F)$ is given by
$$a_{\ga}(0,F)=y^k\sum_{n\ge0}a_{\ga}(n)\ov{b_{\ga}(n)}e^{-4\pi y(\al(c)+n/w(c))}\;,$$
so that
$$\M(a_{\ga}(0;F))(s)=\dfrac{\G(s+k)}{(4\pi)^{s+k}}\sum_{n\ge0}\dfrac{a_{\ga}(n)\ov{b_{\ga}(n)}}{(\al(c)+n/w(c))^{s+k}}\;.$$
We must now be careful about convergence of the Petersson product.
When $k\ge1$, the necessary and sufficient condition is that at every cusp
either $f$ or $g$ vanishes, or equivalently that for every $\ga$
at least one of the forms $f|_k\ga$ and $g|_k\ga$ vanishes at infinity.
In these cases, if $\al(c)=0$ we have necessarily
$a_{\ga}(0)\ov{b_{\ga}(0)}=0$, which means that we omit the term $n=0$, while
if $\al(c)\ne0$ we must keep it.

However the Petersson product also converges without any condition on
$f$ and $g$ if $k=1/2$. In that case, if $\al(c)=n=0$ the contribution to
$a_{\ga}(0;F)$ is $a_{\ga}(0)\ov{b_{\ga}(0)}y^{1/2}$, and although the Mellin
transform is divergent, we will need to take a limit as we will see below.

We deduce from the above corollary and the explicit expression of
$\Lambda(2s)$ the following temporary result:

\begin{proposition} Keep the above assumptions and notation. We have
  \begin{align*}<f,g>_{\G_0(N)}&=\dfrac{1}{[\G:\G_0(N)]}\sum_{c\in C(G)}w(c)\sum_{n\ge0}\dfrac{a_{\ga_c}(n)\ov{b_{\ga_c}(n)}}{(4\pi(\al(c)+n/w(c)))^{k-1}}\cdot\\
    &\phantom{=}\cdot\dfrac{1}{2\pi i}\int_{\Re(s)=1+\de}(4s-2)\dfrac{\G(s)\G(k+s-1)\zeta(2s)}{(4\pi^2(\al(c)+n/w(c)))^s}\,ds\;.
  \end{align*}
\end{proposition}

There are now two ways to continue, and we consider both.

First, we write $\zeta(2s)=\sum_{m\ge1}m^{-2s}$, so that the integral is equal
to the sum from $m=1$ to $\infty$ of the inverse Mellin transform at
$x=4\pi^2m^2(\al(c)+n/w(c))$ of the function $(4s-2)\G(s)\G(k+s-1)$.
Since this inverse Mellin transform is equal to
$$4x^{(k-1)/2}(2x^{1/2}K_{k-2}(2x^{1/2})-K_{k-1}(2x^{1/2}))\;,$$
we obtain our final theorem, due to D.~Collins, although in a slightly
different form:

\begin{theorem}\label{thmnelcol} Let $f$ and $g$ be in $M_k(\G_0(N),\chi)$
  such that either $fg$ vanishes at all cusps or $k=1/2$, and keep all the
  above notation. We have
  \begin{align*}<f,g>_{\G_0(N)}&=\dfrac{4(8\pi)^{-(k-1)}}{[\G:\G_0(N)]}\sum_{c\in C(G)}w(c)\cdot\\
    &\phantom{=}\cdot\sum_{n\ge0}\dfrac{a_{\ga_c}(n)\ov{b_{\ga_c}(n)}}{((\al(c)+n/w(c)))^{k-1}}W_k(4\pi(\al(c)+n/w(c))^{1/2})\;,\end{align*}
  where $W_k(x)=\sum_{m\ge1}(mx)^{k-1}(mxK_{k-2}(mx)-K_{k-1}(mx))$.
  In the special case $k=1/2$, $\al(c)=0$, and $n=0$, the term
  $(n/w(c))^{1/2}W_{1/2}(4\pi (n/w(c))^{1/2}$ is to be interpreted as
  its limit as $n\to0$, in other words as $1/(4(2\pi)^{1/2})$.
\end{theorem}

We will study the function $W_k(x)$ and its implementation below.

But there is another way to continue. Assume for simplicity that $\al(c)=0$
(so that the sum starts at $n=1$). We can write
$$\zeta(2s)\sum_{n\ge1}\dfrac{a_{\ga_c}(n)\ov{b_{\ga_c}(n)}}{n^{s+k-1}}=\sum_{N\ge1}\dfrac{A_{\ga_c}(N)}{N^{s+k-1}}\;,$$
with $$A_{\ga_c}(N)=\sum_{m^2\mid N}m^{2(k-1)}a_{\ga_c}(N/m^2)\ov{b_{\ga_c}(N/m^2)}\;.$$
Once $a_{\ga_c}(n)$ and $b_{\ga_c}(n)$ computed, the computation of
$A_{\ga_c}(N)$ takes negligible time. The advantage is that $\zeta(2s)$ has
disappeared, and we now obtain a formula involving only the term $m=1$
in the definition of $W_k$, i.e., the function
$V_k(x)=x^{k-1}(xK_{k-2}(x)-K_{k-1}(x))$.

When $\al(c)\ne0$ a similar but more complicated formula can easily be
written. Since anyway as we will see the function $W_k(x)$ can be computed
essentially as fast as the function $V_k(x)$, we have not used this other
method.

\subsection{Computation of the Function $W_k(x)$}

First note that $W_k(x)$ is exponentially decreasing at infinity, more
precisely thanks to the corresponding result for the $K$-Bessel function
it is immediate to show that as $x\to\infty$ we have
$$W_k(x)\sim\sqrt{\pi/2}x^{k-1/2}e^{-x}\;.$$

To compute $W_k(x)$ we introduce the simpler function
$U_k(x)=\sum_{m\ge1}(mx)^kK_k(mx)$, and thanks to the recursions
for the $K$-Bessel functions we have $W_k(x)=U_k(x)-(2k-1)U_{k-1}(x)$, so
we must compute $U_k(x)$. We distinguish between $k$ half-integral and
$k$ integral. For $k$ half-integral we have the following easy proposition
which comes from the fact that $K_k$ is an elementary function:

\begin{proposition}\label{propsk} Define polynomials $P_k(x)$ by $P_0(x)=1$
  and the recursion $P_{k+1}(x)=x((k+1)P_k(x)-(x-1)P'_k(x))$ for $x\ge0$, and
  set $S_k(x)=P_k(x)/(x-1)^{k+1}$. For all $k\ge0$ integral we have
  $$U_{k+1/2}(x)=\sqrt{\dfrac{\pi}{2}}\sum_{0\le j\le k}\dfrac{x^{k-j}(k+j)!}{j!(k-j)!2^j}S_{k-j}(e^x)\;.$$
\end{proposition}

This makes the computation of $U_{k+1/2}(x)$ essentially trivial.

\smallskip

We now consider the slightly more difficult problem of computing
$U_k(x)$ when $k$ is integral. Since $K_k(mx)$ tends exponentially fast to
$0$ we could of course simply sum $(mx)^kK_k(mx)$ until the terms become
negligible with respect to the desired accuracy, using the {\tt Pari/GP}
built-in function {\tt besselk} for computing $K$-Bessel functions. But there
is a way which is at least an order of magnitude faster. First note the
following lemma, which comes directly from the integral representation of the
$K$-Bessel function:

\begin{lemma} We have
  $$U_k(x)=\dfrac{x^k}{2}\int_{-\infty}^\infty S_k(e^{x\cosh(t)})\cosh(kt)\,dt\;,$$
  where the functions $S_k$ are as above.
\end{lemma}

Note that as $t\to\pm\infty$ the function $e^{x\cosh(t)}$ tends to
infinity \emph{doubly exponentially}, and since $S_k(X)=P_k(X)/(X-1)^{k+1}$,
the integrand tends to $0$ doubly-exponentially. This is exactly the context of
\emph{doubly exponential integration}, except that here there is no
change of variable to be done. The basic theorem, due to Takahashi and Mori,
states that the fastest way to compute this integral is as a Riemann sum
$h\sum_{-N\le j\le N}R_k(jh)$, where $R_k$ is the integrand and $h$ and $N$
are chosen appropriately (we do not need the theorem since we compute
the errors explicitly, but it is reassuring that we do not have a better way).
An easy study both of the speed of doubly-exponential
decrease and of the Euler--MacLaurin error made in approximating the
integral by Riemann sums gives the following:

\begin{proposition} Set $R_k(x)=S_k(e^{x\cosh(t)})\cosh(kt)$, where
  $S_k$ is given by Proposition \ref{propsk}. Let $B>0$ and set
  $C=B+k\log(x)/\log(2)+1$, $D=C\log(2)+2.06$, $E=2((C-1)\log(2)+\log(k!))/x$,
  $T=\log(E)(1+(2k/x)/E)$, $N=\lceil(T/\pi^2)(D+\log(D/\pi^2))\rceil$, and
  $h=T/N$. There exists a small (explicit) constant $c_k$ such that
  $$|U_k(x)-x^kh(R_k(0)/2+\sum_{1\le j\le N}R_k(hj))|<c_k2^{-B}\;.$$
\end{proposition}

Note that in practice, since we need both, it is faster to compute
$U_k(x)$ and $U_{k-1}(x)$ simultaneously.

\subsection{Conclusion: Comparison of the Methods}

After explaining how to expand $f|_k\ga$ using products of two Eisenstein
series, we have given two methods to compute Petersson products. The first
is limited to integral weight $k\ge2$, while the second is applicable to
any $k$ integral or half-integral. In fact, the second method is applicable
to more general modular forms, for instance to modular forms with multiplier
system of modulus $1$ (such as $\eta(\tau)$ and more generally eta quotients),
since the only thing that we need is that $f(\tau)\ov{g(\tau)}y^k$ be
invariant by some subgroup of $\G$. For instance, this implies formulas such
as
$$\sum_{m\ge1,\ m\equiv\pm1\pmod{6}}\dfrac{m}{e^{2\pi m/\sqrt{6}}-1}=\dfrac{1}{12}\;,$$
which can easily be proved directly.

In both methods we need to compute the Fourier expansions of
$(f|_k\ga_c)_{c\in C}$ for a system of representatives $c$ of cusps,
the Fourier expansions of $f|_k\ga_j$ for a complete system of coset
representatives, necessary in the Haberland case, being trivially obtained
from those. This will be by far the most time-consuming part of the
methods. The computation of the integrals in the Haberland case, or of
the infinite series involving the transcendental function $W_k(x)$ in
the Nelson--Collins case will in fact require little time in comparison.

The main difference between the methods comes from the speed of convergence.
In the Haberland case, we have seen that the convergence is at worse in
$e^{-2\pi n/N}$ (when the width of the cusp is equal to $N$, for instance
for the cusp $0$), and for this to be less than $e^{-E}$, say, we need
$n>(E/(2\pi))N$, proportional to $E$. On the other hand, in the Nelson--Collins
case the convergence is at worse in $e^{-4\pi(n/N)^{1/2}}$, so here we need
$n>(E/(4\pi))^2N$, proportional to $E^2$. Thus this latter method is
considerably slower than the former, especially in high accuracy, hence must
be used only when Haberland-type methods are not applicable, in other words
in weight $1$ and half-integral weight.

As a typical timing example, in level $96$, weight $4$, computing a Petersson
product at 19 decimal digits (using Haberland) requires $1.29$s and at 38
decimal digits $2.27$s. On the other hand, in level $96$ weight $5/2$,
computing a Petersson product at 19 decimal digits (using Nelson--Collins)
requires $3.56$s, but at 38 decimal digits $16.2$s.

\end{document}